\documentclass[12pt]{article}
\usepackage{fullpage}
\usepackage{color}
\usepackage{graphicx}
\usepackage{epstopdf}
\usepackage{amssymb,amsmath,amsthm,xfrac}
\usepackage{yhmath}

\usepackage[mathscr]{euscript}
\usepackage[all]{xy}
\usepackage{hyperref}
\usepackage[nameinlink,capitalize]{cleveref}

\newcommand{\dist}{\hbox{dist}}

\definecolor{todocolor}{rgb}{1.0, 0.0, 0.5}

 \newtheorem{thm}{Theorem} 
 \newtheorem{cor}[thm]{Corollary}
 \newtheorem{lem}[thm]{Lemma}
 \newtheorem{prop}[thm]{Proposition}
 \theoremstyle{definition}
 
 \theoremstyle{remark}

 \numberwithin{equation}{section}

\begin{document}

\title{Related Problems in Spherical and Solid Geometry}
\author{Michael Q. Rieck\footnote{Drake University, Math \& CS Dept., Des Moines, IA 50310, USA; {\tt michael.rieck@drake.edu}}}
\date{ \ } 
\maketitle
\begin{abstract}
A problem that is simple to state in the context of spherical geometry, and that seems rather interesting, appears to have been unexamined to date in the mathematical literature. The problem can also be recast as a problem in the real projective plane. The problem on the sphere involves four great circles and their intersections. A substantial claim is made concerning this problem, and subsequently proved by relating the spherical problem to a compelling problem in solid geometry. This latter problem essentially concerns relationships between the angles of a tetrahedron, and has practical applications, particularly in connection with the Perspective 3-Point (Pose) Problem. 

\vspace{4mm}


\vspace{-2mm}

\noindent {\bf keywords:} spherical geometry, solid geometry, tetrahedron, projective plane, elliptic plane. \\ 

\end{abstract}

\section{A problem in spherical geometry}

\begin{figure}[ht]
  \centering
  \includegraphics[width=6.5cm]{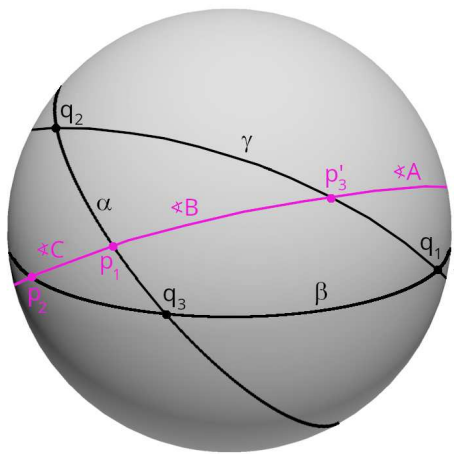} 
  \caption{The spherical geometry problem}
  \label{fig:fig1}
\end{figure}

The problem considered in this section is concerned with great circles on the unit sphere. It begins with six positive numbers, $\angle A$, $\angle B$, $\angle C$, $\alpha$, $\beta$, and $\gamma$, subject to these restrictions: $\angle A + \angle B + \angle C = \pi$, $\alpha + \beta + \gamma < 2 \pi$, $\alpha < \beta + \gamma$, $\beta < \gamma + \alpha$ and $\gamma < \alpha + \beta$. The notation ``$\angle A$" is not here meant to imply the existence of a point ``$A$," and so forth. This notation will be justified in the next section. All angles will be measured in radians. 

Consider any three great circles, with three arcs, one from each circle, forming a spherical triangle with side lengths $\alpha$, $\beta$ and $\gamma$. The constraints on the numbers $\alpha$, $\beta$ and $\gamma$ ensure that such circles exist. The circle with the arc of length $\alpha$ (resp, $\beta$, $\gamma$) will be denoted $\mathscr{C}_1$ (resp. $\mathscr{C}_2$, $\mathscr{C}_3$). For another arbitrary great circle $\mathscr{C}$, let $p_j$ and $p'_j$ denote the intersection points of $\mathscr{C}$ and $\mathscr{C}_j$ ($j=1,2,3$). Without loss of generality, moving around $\mathscr{C}$ (in one of the two directions), assume that the six intersection points occur in this cyclic order: $p_1, p_2, p_3, p'_1, p'_2, p'_3$. See Figure 1. Now ask the following. \\

\noindent {\bf Question 1.} Does there exist a great circle $\mathscr{C}$ along which the arcs from $p_1$ to $p_2$, from $p_2$ to $p_3$, and from $p_3$ to $p'_1$, have lengths $\angle C$, $\angle A$ and $\angle B$, respectively? \\

There is a comparatively simple case with a very simple answer, as follows. 

\begin{prop} 
If $\angle A$, $\angle B$ and $\angle C$ do not exceed $\pi/2$, then Question 1 is answered in the affirmative. 
\end{prop} 

\noindent Before proving this claim, the question posed here will be shown to be equivalent to a different question in three-dimensional real space. 

\section{A problem in solid geometry}

We continue to assume that numbers $\angle A$, $\angle B$, $\angle C$, $\alpha$, $\beta$, and $\gamma$, as in the previous section, have been fixed. Let $\ell_1$, $\ell_2$ and $\ell_3$ be three lines in $\mathbb{R}^3$, passing through the origin, with $\alpha$ and $\pi-\alpha$ being the angles between $\ell_2$ and $\ell_3$, $\beta$ and $\pi-\beta$ being the angles between $\ell_3$ and $\ell_1$, and $\gamma$ and $\pi-\gamma$ being the angles between $\ell_1$ and $\ell_2$. Because of the constraints on $\alpha$, $\beta$ and $\gamma$, such lines exist. Each line $\ell_j$ intersects the unit sphere in two antipodal points, $q_j$ and $q'_j$. The three lines taken pairwise determine three planes through the origin which then intersect the unit sphere in three great circles $\mathscr{C}_1$, $\mathscr{C}_2$, $\mathscr{C}_3$, yielding the preliminary setup described in Section 1. 

\begin{figure}[ht]
\vspace{-6mm}
  \centering
  \includegraphics[width=5cm]{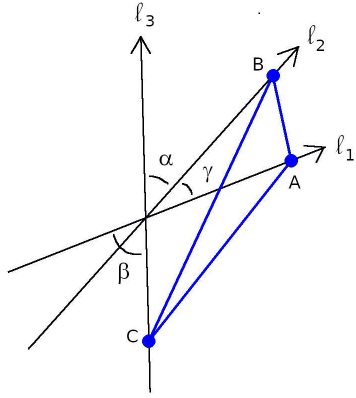} 
  \caption{The solid geometry problem}
  \label{fig:fig2}
\end{figure}

Next consider any three distinct points $A$, $B$ and $C$, with $A$ on $\ell_1$, $B$ on $\ell_2$, and $C$ on $\ell_3$. Consider the triangle $\triangle ABC$. See Figure 2. Now ask the following. \\

\noindent {\bf Question 2.} Can $A$, $B$ and $C$ be selected so that the interior angle of $\triangle ABC$ at $A$ is $\angle A$, the interior angle at $B$ is $\angle B$, and the angle at $C$ is $\angle C$? \\

This leads to a claim somewhat similar to that in Section 1, and also a connection between these two claims. 

\begin{prop} 
If $\angle A$, $\angle B$ and $\angle C$ do not exceed $\pi/2$, then Question 2 is answered in the affirmative. 
\end{prop} 

\begin{prop}

Questions 1 and 2 are equivalent in that both depend only on the numbers $\angle A$, $\angle B$, $\angle C$, $\alpha$, $\beta$, and $\gamma$, and the answers to these questions are the same. Thus, Propositions 1 and 2 are equivalent too. 

\end{prop}

\begin{proof}[Proof of Prop. 3]

Concerning the setup in this section, the sideline $\overleftrightarrow{BC}$ of $\triangle ABC$, when parallel translated so as to pass through the origin, yields a line $\mathscr{L}_1$ that intersects the unit sphere in two antipodal points $p_1$ and $p'_1$. These points are concyclic with $q_2$, $q_3$, $q'_2$ and $q'_3$, with their common circle in the plane containing the lines $\mathscr{L}_1$, $\ell_2$, $\ell_3$. Likewise, $\overleftrightarrow{CA}$ ($\overleftrightarrow{AB}$) yields a line $\mathscr{L}_2$ ($\mathscr{L}_3$) and antipodal points $p_2$ and $p'_2$ ($p_3$ and $p'_3$). 

The six points $p_1$, $p'_1$, $p_2$, $p'_2$, $p_3$ and $p'_3$ lie on a great circle of the unit sphere that is also in a plane parallel to the plane containing the triangle $\Delta ABC$. We may assume that going around this circle, the points occur in the cyclic order $p_1, p_2, p_3, p'_1, p'_2, p'_3$, as before. It is immediately seen that the distance along the arc connecting $p_1$ and $p_2$ equals the interior angle of $\triangle ABC$ at $C$, and so forth. Thus we achieve the complete setup in Section 1, manifesting the claimed connection between the two questions. 

\end{proof}

\section{Extending Sullivan's Construction}

Before proving Proposition 2, a framework for exploring Question 2 will now be developed, without yet assuming that $\angle A$, $\angle B$ and $\angle C$ are bounded above by $\pi/2$. Without loss of generality, assume that $\ell_1$ is the $x$-axis, that $\ell_2$ is also in the $xy$-plane, but $\ell_3$ is not in the $xy$-plane. Further assume that rotating the positive $x$-axis counterclockwise by the angle $\gamma$ results in a ray along the line $\ell_2$. Now, choose a positive number $a$, and let $b = a \csc \angle A \sin \angle B$ and $c = a \csc \angle A \sin \angle C$.

\begin{figure}[ht]
  \centering
  \includegraphics[width=8cm]{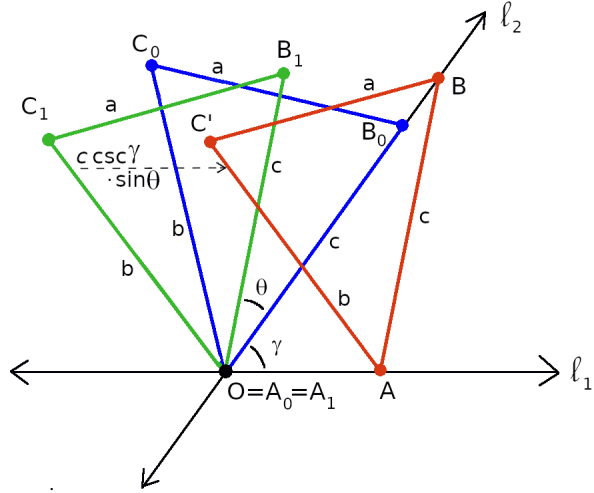} 
  \caption{Sullivan's Construction}
  \label{fig:fig3}
\end{figure}

Let $B_0$ be the point along this ray whose coordinates are $(c \cos \gamma, c \sin \gamma, 0)$. Let $C_0$ be the point whose coordinates are $(b \cos (\angle A+\gamma), b \sin(\angle A+\gamma), 0)$. Let $A_0 = O$ (the origin). A direct check shows that $\angle C_0A_0B_0 = \angle A$ , $\angle A_0B_0C_0 = \angle B$ , $\angle B_0C_0A_0 = \angle C$ , \linebreak dist($B_0$, $C_0$) = $a$ , dist($C_0$, $A_0$) = $b$ and \ dist($A_0$, $B_0$) = $c$. Now, rotate $B_0$ and $C_0$ about the origin by an arbitrary signed angle $\theta$ to get points $B_1$ and $C_1$. Let $A_1 = A_0 = O$. Next, translate $A_1$, $B_1$ and $C_1$ along the $x$ direction by an amount $c \csc \gamma \sin \theta$ to get the points $A$, $B$ and $C'$. Notice that $\angle C'AB = \angle A$ , $\angle ABC' = \angle B$ , $\angle BC'A = \angle C$ , dist($B$, $C'$) = $a$ , dist($C'$, $A$) = $b$ , dist($A$, $B$) = $c$. Also notice that $A$ is on $\ell_1$, $B$ is on $\ell_2$, but $C'$ is not on $\ell_3$, since it is in the $xy$-plane. See Figure 3. 

The construction thus far is discussed in \cite{W}, where it is credited to John M. Sullivan. By varying $\theta$, we will regard that $A$, $B$ and $C'$ are ``moving points." As $A$ is allowed to slide back and forth along $\ell_1$ (crossing $O$), and $B$ does likewise along $\ell_2$, keeping a distance $c$ from $A$, these two points move in a manner reminiscent of the {\it Trammel of Archimedes}, a skewed version of it actually (cf. \cite{W}). This mathematical model will now be extended to obtain a three-dimension method for producing the desired triangle to fit the three given lines. 

Let $F$ be the projection of $C'$ onto the line $\overleftrightarrow{AB}$, and let $C''$ be the reflection of $C'$ about the line $\overleftrightarrow{AB}$. Now, consider rotating $C'$ in three dimensions, about the line $\overleftrightarrow{AB}$ by a signed angle $\psi$ to obtain a point $\wideparen{C}$. When $\psi = 0$, $\wideparen{C} = C'$, of course, but when $\psi = \pi$, $\wideparen{C} = C''$. Moreover, for arbitrary $\psi$, the orthogonal projection of $\wideparen{C}$ onto the $xy$-plane is on the line segment $\overline{C'C''}$. Notice too that the triangles $\Delta A_0B_0C_0$, $\Delta A_1B_1C_1$, $\Delta ABC'$, $\Delta ABC''$ and $\Delta AB\wideparen{C}$ are all conguent. Also, by varying both $\theta$ and $\psi$, we produce all possible triangles $\Delta AB\wideparen{C}$ that are congruent to $\Delta A_0B_0C_0$, and for which $A$ is on $\ell_1$ and $B$ is on $\ell_2$. 

\begin{figure}[ht]
  \vspace{2mm}
  \centering
  \includegraphics[width=7cm]{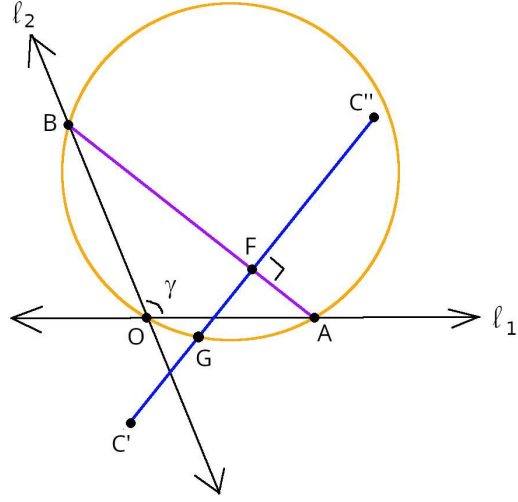} 
  \caption{3D Extension Analysis (in 2D)}
  \label{fig:fig4}
\end{figure}

Now, with $A$ and $B$ fixed at the moment, and with neither being $O$, consider the circle $OAB$, as in Figure 4. At each point $P$ on this circle, other than $A$ or $B$, the angle $\angle APB$ equals either $\gamma$ or $\pi - \gamma$ (by the Inscribed Angle Theorem). If we now slide $A$ and $B$, by continuously changing $\theta$, the same rigid motion that moves the segment $\overline{AB}$ will also move the circle in such a way that the resulting circle will still pass through all three of the points $O$, $A$ and $B$, even though $O$ is stationary. The formula for this rigid motion is $(x,y) \mapsto (x \cos\theta - y \sin\theta + c \csc\gamma \sin\theta \, , \, x \sin\theta + y \cos\theta)$. Moreover, the line segment $\overline{C'C''}$ also moves according to this same rigid transformation. 

Though they will not be needed in the rest of this paper, here are the coordinates of the points discussed thus far:

\vspace{3mm}

{\small $\begin{array}{ccl}
B_0 & : \; & ( \, c \cos \gamma , \; c \sin \gamma , \; 0 \, ) \\ 
C_0 & : \; & ( \, b \cos(\angle A + \gamma) , \; b \sin(\angle A + \gamma) , \; 0 \, ) \\ 
B_1 & : \; & ( \, c \cos(\gamma+\theta) , \; c \sin(\gamma+\theta) , \; 0 \, ) \\ 
C_1 & : \; & ( \, b \cos(\angle A +\gamma + \theta) , \; c \sin(\angle A + \gamma+\theta) , \; 0 \, ) \\ 
A & : \; & ( \, c \csc \gamma \sin \theta , \; 0 , \; 0 \, ) \\ 
B & : \; & ( \, c \cos(\gamma+\theta) + c \csc \gamma \sin \theta , \; c \sin(\gamma+\theta) , \; 0 \, ) \\ 
C' & : \; & ( \, b \cos(\angle A + \gamma+\theta) + c \csc \gamma \sin \theta , \; b \sin(\angle A + \gamma+\theta) , \; 0 \, ) \\ 
C'' & : \; & ( \, b \cos(-\angle A + \gamma+\theta) + c \csc \gamma \sin \theta , \; b \sin(-\angle A + \gamma+\theta) , \; 0) \\ 
\wideparen{C} & \; : \; & ( \, b \cos\angle A \cos(\gamma+\theta) + c \csc \gamma \sin \theta - b \sin \angle A \sin(\gamma+\theta) \cos \psi, \\ 
\ & \ & ( \, b \cos\angle A \sin(\gamma+\theta) + b \sin \angle A \cos(\gamma+\theta) \cos \psi, \; b \sin A \sin \psi \, ) \\ 
F & \; : \; & ( b \cos\angle A \cos(\gamma+\theta) + c \csc \gamma \sin \theta, \;  b \cos\angle A \sin(\gamma+\theta), \; 0 \, ). \\ 
\end{array}$ } 

\vspace{3mm} 

\noindent These coordinates could be useful in better understanding the framework, and perhaps in implementing demonstrations of it. Also, the circle containing $O$, $A$ and $B$ has the equation $x^2 + y^2 = c \csc\gamma \, (\sin\theta \, x + \cos\theta \, y)$. A couple lemmas will be helpful, primarily in proving Proposition 2, which is concerned with the existence of a certain acute or right $\Delta ABC$, but these lemmas also provide insight into the obtuse $\Delta ABC$ case.   \\

\begin{lem}
The following are equivalent: 
\begin{itemize}
\item[(i)] There exist values for $\theta$ and $\psi$ such that $\wideparen{C}$ lies on the $z$-axis; 
\item[(ii)] The line segment $\overline{C'C''}$ intersects the circle $OAB$ (for any value of $\theta$). 
\end{itemize}
\noindent The following are also equivalent: 
\begin{itemize} 
\item[(iii)] $F$ is on or inside the circle $OAB$, and $\angle C \le \max\{\gamma, \pi-\gamma\}$;
\item[(iv)] $F$ is on or inside the circle $OAB$, and $\cos \gamma \le \cos \angle C$ or $\cos \gamma \ge -\cos \angle C$. 
\end{itemize}
\noindent Moreover, (iii) implies (ii). 
\end{lem} 

\begin{proof} 

\vspace{3mm} 

As $\theta$ is continuously increased (or decreased), the circle $OAB$ and the line segments $\overline{AB}$ and $\overline{C'C''}$ will all undergo the same continuous rigid motion. Suppose for a moment that $\overline{C'C''}$ intersects the circle at a point $G$. As the circle and line segments move, the point $G$ moves with them, and follows an elliptic path (by a theorem of Schooten; see \cite{W}) in the $xy$-plane. It must also remain concyclic with $O$, $A$ and $B$. 

Consider now a reference frame (rigidly moving coordinate system) in which $A$, $B$, $C'$, $C''$, $F$ and $G$ are stationary. Here $O$ becomes a moving point (no longer regarded as the origin), and moves around the circle $ABG$.\footnote{It moves twice around this circle each time the objects go around once in the original 3-space frame.} Clearly, $O$ passes through $A$, and passes through $B$, and these events alternate (strictly). Thus it is clear that $O$ passes through $G$ as well. That is, there exists a value for $\theta$ that makes $G = O$. Since $G$ lies on the line segment $\overline{C'C''}$, it is the orthogonal projection of some $\wideparen{C}$ onto the $xy$-plane. Since $G = O$, this $\wideparen{C}$ is on the $z$-axis of the original reference frame. Therefore, ({\it ii}) implies ({\it i}). (The inverse of the rigid motion discussed earlier is $(x,y) \mapsto ( \; (x-c\csc\gamma\sin\theta) \cos\theta + y \sin\theta + c \csc\gamma \sin\theta \, , \, -(x-c\csc\gamma\sin\theta) \sin\theta + y \cos\theta \; )$, which can be used to describe the path now followed by the moving point $O$, by setting $x=y=0$.)

Conversely, assuming ({\it i}), the orthogonal projection of this $\wideparen{C}$ onto the $xy$-plane is $O$. Thus, $O$ must lie on $\overline{C'C''}$, and obviously it lies on the circle $OAB$, and hence $\overline{C'C''}$ and the circle $OAB$ intersect. Since these objects move together as $\theta$ varies, they will still intersect, for any value of $\theta$. Thus, ({\it i}) implies ({\it ii}). The fact that ({\it iii}) implies ({\it ii}) follows after observing that the circle $OAB$ consists of the points $P$ for which $\angle APB$ is either $\gamma$ or $\pi - \gamma$, so either $C'$ or $C''$ is outside the circle. The equivalence of ({\it iii}) and ({\it iv}) is straightforward to check.

\end{proof} 

\begin{lem}
If condition $(ii)$ in Lemma 4 holds, then points $A$, $B$ and $C$, as described in Proposition 2, exist. \\
\end{lem}

\begin{proof}

If $\ell_3$ is the $z$-axis, then the claim here follows immediately from Lemma 4, by simply choosing $C$ to be the $\wideparen{C}$ on the $z$-axis. Now, assume that $\ell_3$ is not the $z$-axis. Express the line $\ell_3$ parametrically as $x = m z$ and $y = n z$, for constants $m$ and $n$, not both zero. Again, for some $\theta$ and some $\psi$, the corresponding $\wideparen{C}$ is along the $z$-axis. Keeping this $\theta$ fixed, vary $\psi$ to get a vertical circle ({\it i.e.} orthogonal to the $xy$-plane) of points $\wideparen{C}$. The line $\ell_3$ either intersects this circle (if it is in the same plane as the circle), or passes through the circle. In the former case, we will have again found a suitable choice for $C$. 

Now assume the latter case. If we now change $\theta$, the vertical circle of possible $\wideparen{C}$ can be made to lie in a vertical plane that is parallel or equal to the vertical plane containing $\ell_3$. To see this, consider the orthogonal projections onto the $xy$-plane. The circle of possible $\wideparen{C}$ projects onto the segment $\overline{C'C''}$. The line $\ell_3$ projects onto a line $\ell$ in the $xy$-plane, given by the equation $n x = m y$. The segment $\overline{C'C''}$ makes a complete rotation as $\theta$ varies, and so it can be made either to lie on $\ell$ or to be parallel to $\ell$. In the former case, we would again have $\ell_3$ intersecting the circle of possible $\wideparen{C}$, and so again a suitable point $C$ would be obtained. 

Now assume the latter case. For this value of $\theta$, the line $\ell_3$ does not pass through the vertical circle. But for the value of $\theta$ considered earlier, $\ell_3$ did pass through the circle. By continuity, there must be a value of $\theta$ for which $\ell_3$ intersects the circle, once again giving us a point $\wideparen{C}$ that will serve as a suitable choice for $C$. 

\end{proof}

We are now prepared to prove Proposition 2, which together with Proposition 3, also will establish Proposition 1. 

\begin{proof}[Proof of Prop. 2]  

Notice that $\angle A \le \pi /2$ and $\angle B \le \pi /2$ ensure that the segments $\overline{AB}$ and $\overline{C'C''}$ intersect, at $F$.  Since $\overline{AB}$ is a chord of the circle $OAB$, $F$ is in the interior of the circle, unless $F = A$ or $F = B$. Since $\angle C \le \pi /2$, it follows that either $\angle C \le \gamma$ or $\angle C \le \pi - \gamma$. The proposition then follows quickly from the two lemmas. 

\end{proof}

\section{Afterthoughts}

Suppose that points $A$, $B$ and $C$ are selected points on $\ell_1$, $\ell_2$ and $\ell_3$, respectively, as in Section 2. If the three lines happen to be orthogonal to each other, then it is easily seen that the triangle $\Delta ABC$ cannot be obtuse, and must instead be an acute or a right triangle. Proposition 2 ensures that any acute or right triangle is possible here. Another immediate consequence of Proposition 2 is as follows. 

\begin{cor}

A tetrahedron with three prescribed edge lines at one of its vertices can have an opposite face that is congruent to any given acute or right triangle.

\end{cor}

Extensive experiments using Mathematica and C++ seem to indicate how to extend Proposition 2 to the obtuse triangle case, where one of the numbers $\angle A$, $\angle B$, $\angle C$ exceeds $\pi/2$. Also, a practical variation of Question 2 is obtained by replacing the three lines, $\ell_1$, $\ell_2$ and $\ell_3$, with three rays from the origin. This revised question is closely coupled with the ``Three-Point Perspective (Pose) Problem" discussed in \cite{HLON} and \cite{RW}. Not surprisingly, in this variation of Question 2, the acute triangle case and the obtuse triangle case again yield quite different results. Thorough testing has resulted in several proven claims and a number of conjectures. The ``discriminant of Grunert's system" and the ``basic toroids" considered in \cite{RW} play crucial roles in the analysis.

Finally, Question 1 and Proposition 1 can be recast in terms or the real projective plane $\mathbb{RP}^2$. This is commonly modeled by taking its ``points" to actually be pairs of antipodal points on the real unit sphere. Another useful model identifies the points of the projective plane with the lines through the origin in $\mathbb{R}^3$. The correspondence between these two models is evident since each line through the origin intersects the sphere in a pair of antipodal points. 

The ``elliptic plane" is a name often given to the manifold that results from equipping $\mathbb{RP}^2$ with the metric inherited from the unit sphere. (The name comes from ``elliptic geometry," which confusingly has little to do with ellipses.) We will regard that $\mathbb{RP}^2$ is equipped with this metric. Using the reasoning here, the following claim can be quickly seen to be a consequence of Proposition 1. \\

\begin{cor}
Let $P_1$, $P_2$ and $P_3$ be three collinear points on the elliptic plane such that $\dist(P_2, P_3) \le \pi/2$, $\dist(P_3, P_1) \le \pi/2$, $\dist(P_1, P_2) \le \pi/2$, and $\dist(P_2, P_3) + \dist(P_3, P_1) + \dist(P_1, P_2) = \pi$. Also, let $\alpha$, $\beta$ and $\gamma$ be any strictly positive numbers with $\alpha + \beta + \gamma < 2 \pi$, $\alpha < \beta + \gamma$, $\beta < \gamma + \alpha$ and $\gamma < \alpha + \beta$. Then, there exist points $Q_1$, $Q_2$ and $Q_3$ with $P_1, Q_2, Q_3$ collinear, $P_2, Q_3, Q_1$, collinear, $P_3, Q_1, Q_2$ collinear, $\dist(Q_2, Q_3) = \min\{\alpha,\pi-\alpha\}$, $\dist(Q_3, Q_1) = \min\{\beta, \pi-\beta\}$, and $\dist(Q_1, Q_2) = \min\{\gamma, \pi-\gamma\}$.  \\
\end{cor}

\noindent This basically means that whenever three collinear points $P_1$, $P_2$ and $P_3$ in $\mathbb{RP}^2$ are not ``bunched too close together," it is always possible to construct a triangle in $\mathbb{RP}^2$ (the projection of a spherical triangle onto $\mathbb{RP}^2$), of any possible shape, such that each of its sidelines is incident, in any prescribed order, with one of the three collinear points. \\



\begin{thebibliography}{10}

\bibitem[Haralick et~al., 1994]{HLON}
Haralick, R.~M., Lee, C.-N., Ottenberg, K., and N{\"o}lle, N. 
\newblock Review and Analysis of Solutions of the Three Point Perspective Pose Estimation Problem.
\newblock {\em J. Computer Vision}, 13(3): 331--356.

\bibitem[Rieck-Wang, 2021]{RW}
Rieck, M.~Q., and Wang, B. 
\newblock Locating Perspective Three-Point Problem Solutions in Spatial Regions.
\newblock {\em J. Mathematical Imaging and Vision}, 63(8): 953--973.

\bibitem[Wetzel, 2010]{W}
Wetzel, J.~E. 
\newblock An Ancient Elliptic Locus. 
\newblock {\em Amer. Math. Monthly}, 117(2): 161--167. 

\end{thebibliography}
\end{document}